\def\misajour{March 19, 2013}

\documentclass[graybox]{svmult}

\usepackage{mathptmx} 
\usepackage{helvet} 
\usepackage{courier} 
\usepackage{type1cm} 
\usepackage{makeidx} 
\usepackage{graphicx} 
\usepackage{multicol} 
\usepackage[bottom]{footmisc}

\usepackage[applemac]{inputenc}

\makeindex 

\def\C{\mathbf{C}}
\def\N{\mathbf{N}}
\def\Q{\mathbf{Q}}
\def\R{\mathbf{R}}
\def\Z{\mathbf{Z}}
\def\ZK{\Z_K}

\def\mubar{\overline{\mu}}
\def\rmh{{\mathrm{h}}}
\def\Ftilde{\tilde{F}}
\def\rmN{{\mathrm N}}

\def\house#1{\setbox1=\hbox{$\,#1\,$}%
\dimen1=\ht1 \advance\dimen1 by 2pt \dimen2=\dp1 \advance\dimen2 by 2pt
\setbox1=\hbox{\vrule height\dimen1 depth\dimen2\box1\vrule}%
\setbox1=\vbox{\hrule\box1}%
\advance\dimen1 by .4pt \ht1=\dimen1
\advance\dimen2 by .4pt \dp1=\dimen2 \box1\relax}

\def\sigmabar{\overline{\sigma}}
\def\epsilonprime{\epsilon'}
\def\xiprime{\xi'}
\def\epsilonprimebar{{\overline{\epsilonprime}}}

\def\xiprimeunbar{{\overline{\xiprime_1}}}
\def\alphaprime{\alpha'}
\def\alphaprimebar{{\overline{\alphaprime}}}
\def\alphatilde{\tilde{\alpha}}
\def\alphaprimetilde{\tilde{\alpha}'}
\def\alphaprimebartilde{\overline{\tilde{\alpha}'}}

\newtheorem{corollaire}{Corollary}
\newtheorem{lemme}{Lemma}

\newcounter{compteurkappa}

\def\Newcst#1{
\refstepcounter{compteurkappa}
\kappa_{
\arabic{compteurkappa}}
\label{#1}
}

\def\cst#1{\kappa_{\ref{#1}}}

\begin{document}

\title*{
Families of cubic Thue equations 
\\
with effective bounds for the solutions 
}
\author{Claude LEVESQUE and Michel WALDSCHMIDT}
\institute{Claude LEVESQUE \at
D\'{e}partement de math\'{e}matiques et de statistique,
 Universit\'{e} Laval,
Qu\'{e}bec (Qu\'{e}bec),
CANADA G1V 0A6
\\
 \email{Claude.Levesque@mat.ulaval.ca}
\and Michel WALDSCHMIDT\at
Institut de Math\'{e}matiques de Jussieu,
Universit\'{e} Pierre et Marie Curie (Paris 6),
4 Place Jussieu,
F -- 75252 PARIS Cedex 05, FRANCE
\\
 \email{miw@math.jussieu.fr}
\\
\null \hfill \it Mise à jour: \bf \misajour}
%
%
\maketitle

\abstract{. To each non totally real cubic extension $K$ of $\Q$
 and to each generator $\alpha$ of the cubic field $K$,
 we attach a family of cubic Thue equations,
 indexed by the units of $K$, and we prove that this
 family of cubic Thue equations has only
 a finite number of integer solutions, by
 giving an effective upper bound for these solutions.}

\section{Statements}
\label{S:Enonces} Let us consider an irreductible binary cubic form
having rational integers coefficients
$$
F(X,Y)\;=\;a_0X^3+a_1X^2Y+a_2XY^2+a_3Y^3\in\Z[X,Y]
$$
with the property that the polynomial $F(X,1)$ has exactly one
real root $\alpha$ and two
 complex imaginary roots, namely $\alphaprime$ and
 $\alphaprimebar$.
 Hence $\alpha\not\in\Q$, $\alphaprime\; \not=\; \alphaprimebar$ and
$$
F(X,Y)\;=\;a_0(X-\alpha Y)(X-\alphaprime Y)(X-\alphaprimebar Y).
$$
Let $K$ be the cubic number field $\Q(\alpha)$ which we view as a
subfield of $\R$. Define $\sigma:K\rightarrow\C$ to be one of the
two complex embeddings, the other one being the conjugate
$\sigmabar$. Hence $\alphaprime\;=\;\sigma(\alpha)$ and
$\alphaprimebar\;=\;\sigmabar(\alpha)$. If $\tau$ is defined to
 be the complex conjugation, we have
 $\sigmabar\;=\;\tau\circ\sigma$ and $\sigma\circ\tau\;=\;\sigma$.

 Let $\epsilon$ be a 
 unit $>1$ of the ring
$\ZK$ of algebraic integers of $K$ and let
$\epsilonprime\;=\;\sigma(\epsilon)$ and
$\epsilonprimebar\;=\;\sigmabar(\epsilon)$ be the two other algebraic
conjugates of $\epsilon$. We have
$$
|\epsilonprime|\;=\;|\epsilonprimebar|\;=\;\frac{1}{\sqrt{\epsilon}}<1.
$$
For $n\in\Z$, define
$$
F_n(X,Y)\;=\;a_0\bigl(X- \epsilon^n\alpha Y\bigr) \bigl(X-
\epsilonprime^n\alphaprime Y\bigr) \bigl(X- \epsilonprimebar^n\,
\alphaprimebar\, Y\bigr).
$$
Let $k\in\N$, where $\N\;=\;\{1,2,\dots\}$. We plan to study the
family of Thue inequations
\begin{equation}\label{Eq:InegaliteThue}
0\;<\;|F_n(x,y)|\;\le\; k,
\end{equation}
where the unknowns $n,x,y$ take values in $\Z $.

\begin{theorem}\label{Theoreme:principal} 
There exist effectively
computable positive constants $\Newcst{kappa1}$ and
$\Newcst{kappa2}$, depending only on $F$,
 such that, for all $k\in\Z$ with $k\ge 1$ and for all
 $(n,x,y)\in\Z\times \Z\times\Z$
 satisfying $\epsilon^n \alpha\not\in \Q$, $xy\;\not=\;0$
and $|F_n(x,y)|\le k$, we have
$$
\max\left\{ \epsilon^{|n|}, \; |x|, \; |y|\right\}\;\le\;
\cst{kappa1} k^{\cst{kappa2}}.
$$
\end{theorem}

 From this theorem, we deduce the following corollary.

\begin{corollaire}.
\label{Corollaire:finitude} For $k\in\Z$, $k>0$, the set
$$
\bigl\{(n,x,y)\in\Z\times \Z\times\Z\; \mid \; \epsilon^n
\alpha\not\in \Q\; ;
 \; xy\not=0\; ;\; |F_n(x,y)|\le k\bigr\}
$$
is finite.
\end{corollaire}

 This corollary is a particular case of the main result of
\cite{LW1}, but the proof in \cite{LW1} is based on the
 Schmidt subspace theorem which does not allow to give an effective
 upper bound for the solutions $(n,x,y)$.
 \\
 
\noindent
{\bf Example.} 
Let $D\in\Z$, $D\not=-1$. Let $\epsilon:\;=\;
\bigl(\root 3 \of {D^3+1} -D\bigr)^{-1}$.
There exist two positive effectively computable absolute constants $\Newcst{kappaExample1}$ and $\Newcst{kappaExample2}$ with the following property. 
Define a sequence $(F_n)_{n\in\Z}$ of cubic forms in $\Z[X,Y]$ by 
$$
F_n(X,Y)\;=\;X^3+a_nX^2Y+b_nXY^2-Y^3,
$$
where $(a_n)_{n\in\Z}$ is defined by the recurrence relation
$$
a_{n+3}\;=\;  3Da_{n+2}+3D^2a_{n+1}+a_n
$$
with the initial conditions $a_0\;=\; 3D^2$, $a_{-1}\;=\; 3$ and $a_{-2}\;=\; -3 D$, and where $(b_n)_{n\in\Z}$ is defined by $b_n\;=\;-a_{-n-2}$.
Then, for $x$, $y$, $n$ rational integers with $xy\not=0$ and $n\not=-1$, we have 
$$
|F_n(x,y)|\ge \cst{kappaExample1} \max\{|x|, \; |y|, \; \epsilon^{|n|}\}^{\cst{kappaExample2}}.
$$

\medskip
This result follows from Theorem 
$\ref{Theoreme:principal}$ with $\alpha=\epsilon$ and 
$$
F(X,Y) \;=\;X^3-3DX^2Y-3D^2XY^2-Y^3.
$$ 
Indeed, the irreducible polynomial of $\epsilon^{-1}=\root 3 \of {D^3+1} -D$ is
$$ 
F_{-2}(X,1)\;=\;(X+D)^3-D^3-1 \;=\; X^3+3DX^2+3D^2X-1,
$$
the irreducible polynomial of $\alpha\; =\; \epsilon$ is 
$$
F(X,1)\;=\; F_0(X,1)\;=\; F_{-2}(1,X) \;=\; X^3-3D^2X^2-3DX-1, 
$$
while
$$
 F_{-1}(X,Y)\;=\;(X-Y)^3\;=\;X^3-3X^2Y+3XY^2-Y^3.
$$
For $n\in\Z$, $n\not=-1$,
 $F_n(X,1)$ is the irreducible polynomial of $\alpha\epsilon^n \;=\; \epsilon^{n+1}$, 
while for any $n\in\Z$, $F_n(X,Y)\;=\;\rmN_{\Q(\epsilon)/\Q}(X- \epsilon^{n+1}Y)$.
The recurrence relation for 
$$
a_n\; \; =\epsilon^{n+1}+{\epsilonprime}^{n+1}+{\epsilonprimebar}^{n+1}
$$
 follows from 
$$
\epsilon^{n+3}\;=\;  3D\epsilon^{n+2}+3D^2\epsilon^{n+1}+\epsilon^n
$$
and for $b_n$, from 
$F_{-n}(X,Y)\;=\; - F_{n-2}(Y,X)$. 
\bigskip\bigskip

\section{Elementary estimates }\label{S:EstimationsElementaires}

For a given integer $k>0$, we consider a solution $(n,x,y)$ in
$\Z^3$ of the Thue inequation ($\ref{Eq:InegaliteThue}$) with
$\epsilon^n\alpha$ irrational and $xy\not=0$. We will use
$\kappa_5,\kappa_6,\dots,\cst{kappa48}$ to designate some constants
depending only on $\alpha$.

Let us firstly explain that in order to prove Theorem 
$\ref{Theoreme:principal}$, we can assume $n\ge 0$ by eventually
permuting $x$ and $y$. Let us suppose that $n<0$ and write
$$
F(X,Y)\;=\;a_3(Y-\alpha^{-1} X)(Y-\alphaprime^{-1}
X)(Y-\alphaprimebar^{-1} X).
$$
Then
$$
F_n(X,Y)\;=\;a_3\bigl(Y- \epsilon^{|n|}\alpha^{-1} X\bigr)
 \bigl(Y- \epsilonprime^{|n|}\alphaprime^{-1} X\bigr)
 \bigl(Y- \epsilonprimebar^{|n|}\, \alphaprimebar^{-1} \, X\bigr).
$$
Now it is simply a matter of using the result for $|n|$ for the
polynomial $G(X,Y)\;=\;F(Y,X)$.

Let us now check that, in order to prove the statements of
\S\ref{S:Enonces}, there is no restriction in assuming that
$\alpha$ is an algebraic integer and that $a_0\;=\;1$. To achieve
this goal, we define
$$
\Ftilde(T,Y)\;=\; T^3+a_1T^2Y+a_0a_2TY^2+a_0^2a_3Y^3\in\Z[T,Y],
$$
so that $a_0^2F(X,Y)\;=\;\Ftilde(a_0X,Y)$. If we define\,
$\alphatilde\;=\;a_0\alpha$ \,and \,
$\alphaprimetilde\;=\;a_0\alphaprime$,\, then $\alphatilde$ is a nonzero
algebraic integer, and we have
$$
\Ftilde(T,Y)\;=\;(T-\alphatilde Y)(T-\alphaprimetilde
Y)(T-\alphaprimebartilde \, Y).
$$
For $n\in\Z$, the binary form
$$
\Ftilde_n(T,Y)\;=\;(T-\epsilon^n \alphatilde Y)(T-\epsilonprime^n\,
\alphaprimetilde Y) (T-\epsilonprimebar^n \, \alphaprimebartilde\,
Y)
$$
satisfies
$$
a_0^2F_n(X,Y)\;=\;\Ftilde_n(a_0X,Y).
$$
The condition ($\ref{Eq:InegaliteThue}$) implies\,
$0<|\Ftilde_n(a_0x,y)|\le a_0^2 k$.\, Therefore it suffices to
prove the statements for $\Ftilde_n$ instead of $F_n$, with
$\alpha$ and $\alphaprime$ replaced by $\alphatilde$ and
$\alphaprimetilde$. This allows us, from now on, to suppose
$\alpha\in\ZK$ and $a_0\;=\;1$.

As already explained, we can assume $n\ge 0$. There is no
restriction in suppo\-sing $k\ge 2$; (if we prove the result for
a value of $k\geq 2$, we deduce it right away for smaller values
of $k$, since we consider Thue inequations and not Thue
equations). If $k$ were asumed to be $\geq 2$, we would not need $
\cst{kappa1}$, as is easily seen, and the conclusion would read
$$
\max\{\epsilon^{|n|} , \; |x|, \; |y|\}\;\le\; k^{\cst{kappa2}}.
$$
Without loss of generality we can assume that $n$ is sufficiently
large. As a matter of fact, if $n$ is bounded, we are led to some
given Thue equations, and Theorem $\ref{Theoreme:principal}$
follows from Theorem 5.1 of \cite{ST}.

Let us recall that for an algebraic number $\gamma $, the house
of $\gamma$, denoted $\house{\gamma} $\,, is by definition the
maximum of the absolute values of the conjugates of $\gamma$. 
Moreover, $d$ is the degree of the algebraic number field $K$ 
(namely $d\;=\;3$ here) and $R$ is the regulator of $K$ 
(viz. $R=\log \epsilon$), where, from now on,  $\epsilon$ is the fundamental unit $> 1$ of the non totally real cubic
 field $K$.
The next statement is Lemma A.6 of \cite{ST}.

\begin{lemme}\label{Lemme: CorollaryA.6ST}
 Let $\gamma$ be a nonzero element of $\ZK$ of norm $\le M$.
 There exists a unit $\eta\in\ZK^\times$ such that the house
 $\house{\eta\gamma}$ is bounded by an effectively computable
 constant which depends only on $d$, $R$ and $M$.
 \end{lemme}

We need to make explicit the dependence upon $M$, and for this,
it suffices to apply Lemma A.15 of \cite{ST}, which we want to
state, under the asumption that the $d$ embeddings of the
algebraic number field $K$ in $\C$ are noted
$\sigma_1,\ldots,\sigma_d$.

 \begin{lemme}\label{Lemme: LemmaA.15ST}
 Let $K$ be an algebraic number field of degree $d$ and
 let $\gamma$ be a nonzero element of $\ZK$ whose absolute
 value of the norm is $m$.
 Then there exists a unit $\eta \in\ZK^\times$ such that
 $$
 \frac{1}{R}\;
 \max_{1\le j\le d}
 \left|
 \log (m^{-1/d}
 |\sigma_j(\eta\gamma)|)
 \right|
 $$
 is bounded by an effectively computable constant which
 depends only on $d$.
 \end{lemme}

 Since $d\;=\;3$, $K\;=\;\Q(\alpha)$ and the regulator $R$ of $K$ is an effectively computable constant (see for instance \cite{Cohen}, \S 6.5),
 the conclusion of Lemma $\ref{Lemme: LemmaA.15ST}$ is
 $$
 -\cst{kappa3}\; \le\;
 \log (
 |\sigma_j(\eta\gamma)|/{\root 3 \of m} )\;\le\; \Newcst{kappa3},
 $$
which can also be written as
 $$
 \cst{kappa4} {\root 3 \of m} \;\le\;
 |\sigma_j(\eta\gamma)|\;\le\; \cst{kappa5} {\root 3 \of m},
 $$
with two effectively computable positive constants $\Newcst{kappa4}$
and $\Newcst{kappa5}$. We will
 use only the upper bound \footnote{The
lower bound follows from looking at the norm!}: under the
hypotheses of Lemma $\ref{Lemme: CorollaryA.6ST}$ with $d\;=\;3$,
 when $\gamma$ is a nonzero element of $\ZK$ of norm $\le
M$, there exists a unit $\eta$ of $\ZK^\times$ such that
$$
 \house{\eta\gamma}\; \le\; \cst{kappa5} {\root 3 \of M}.
 $$

Since $(n,x,y)$ satisfies ($\ref{Eq:InegaliteThue}$), the element
$\gamma\;=\;x-\epsilon^n \alpha y$ of $\ZK$ has a norm of absolute
value $\le k$. It follows from Lemma $\ref{Lemme: LemmaA.15ST}$
that $\gamma$ can be written as
\begin{equation}\label{Equation:PartieReelle}
x-\epsilon^n \alpha y\;=\;\epsilon^\ell \xi_1
\end{equation}
with $\ell\in\Z$, $\xi_1\in\ZK$ and the house of $\xi_1$,
$\house{\xi_1} \;=\;\max\{|\xi_1|,\; |\xiprime_1|\}$, satisfies
$$\house{\xi_1}\;\le \; \Newcst{kappa6} {\root 3 \of k}.$$
We will not use the full force of this upper bound, but only the consequence
\begin{equation}\label{Equation:xiun}
\max\left\{|\xi_1|^{-1}, |\xiprime_1|^{-1},
\house{\xi_1}\right \}\;\le\; k^{\Newcst{kappaxi}}.
\end{equation}
Taking the conjugate of ($\ref{Equation:PartieReelle}$) by
$\sigma$, we have
\begin{equation}\label{Equation:PremierConjugue}
x-\epsilonprime^n \alphaprime y\;=\; \epsilonprime^\ell \xiprime_1
\end{equation}
with $ \xiprime_1\;=\;\sigma(\xi_1)$.

Our strategy is to prove that $|\ell|$ is bounded by a constant times $\log k$, and that $|n|$ is also bounded by a constant times $\log k$; then we will show that $|y|$ is bounded by a a constant power of $k$ and deduce that $|x|$ is also bounded by a constant power of $k$.
 
Let us eliminate $x$ in ($\ref{Equation:PartieReelle}$) and
($\ref{Equation:PremierConjugue}$) to obtain
\begin{equation}\label{Equation:Y}
y\;=\;-\frac{\epsilon^\ell \xi_1 - \epsilonprime^\ell \xiprime_1}
{\epsilon^n \alpha- \epsilonprime^n \alphaprime };
\end{equation}
since we supposed $\epsilon^n\alpha$ irrational, we did not divide
by $0$. The complex conjugate of
 ($\ref{Equation:PremierConjugue}$) is written as
\begin{equation}\label{Equation:DeuxiemeConjugue}
x-\epsilonprimebar^n\, \alphaprimebar y\;=\;\epsilonprimebar^\ell\,
\xiprimeunbar.
\end{equation}
We eliminate $x$ and $y$ in the three equations
($\ref{Equation:PartieReelle}$), ($\ref{Equation:PremierConjugue}$)
and ($\ref{Equation:DeuxiemeConjugue}$) to obtain a unit
equation \`a la Siegel:
\begin{equation}\label{Equation:EquationAuxUnites}
\epsilon^\ell \xi_1 (\alphaprime \epsilonprime^n-\alphaprimebar \;
\epsilonprimebar^n)
 + \epsilonprime^\ell \xiprime_1 (\alphaprimebar \; \epsilonprimebar^n-\alpha\epsilon^n)
 + \epsilonprimebar^\ell \, \xiprimeunbar(\alpha\epsilon^n-
 \alphaprime\epsilonprime^n) \;=\;0.
\end{equation}

In the remaining part of this section
$\ref{S:EstimationsElementaires}$, we suppose
\begin{equation}\label{Equation:nGrand}
\epsilon^n |\alpha|\; \ge\; 2| \epsilonprime^n \alphaprime|.
\end{equation}
Note that if this inequality is not satisfied, then we have
$$
\epsilon^{3n/2}\; < \ \frac{2|\alphaprime|}{|\alpha|} \; <  \Newcst{kappa:epsilonpuissancen},
$$
and this leads to the inequality ($\ref{Equation:Majoration-n}$),
and to the rest of the proof of Theorem $\ref{Theoreme:principal}$
by using the argument following the inequality
($\ref{Equation:Majoration-n}$).

For $\ell>0$, the absolute value of the numerator
$\epsilon^\ell \xi_1 - \epsilonprime^\ell \xiprime_1$ in
$(\ref{Equation:Y}$) is increasing like $\epsilon^\ell$ and for
$\ell<0$ it is increasing like $\epsilon^{|\ell|/2}$; for $n>0$,
 the absolute value of the denominator $\epsilon^n \alpha-
\epsilonprime^n \alphaprime$ is increasing like $\epsilon^n$ and
for $n<0$ it is increasing like $\epsilon^{|n|/2}$. In order to
extract some information from the equation ($\ref{Equation:Y}$), we
write it in the form
$$
y\;=\;\pm \frac{A-a}{B-b}
$$
with
$$
B\;=\; \epsilon^n \alpha, \quad b\;=\; \epsilonprime^n \alphaprime, \qquad
\{A,a\}\;=\;\left\{ \epsilon^\ell \xi_1 \; , \; \epsilonprime^\ell
\xiprime_1 \right\},
$$
the choice of $A$ and $a$ being dictated by
$$
|A| \;=\; \max\{\epsilon^\ell |\xi_1| \; , \; |\epsilonprime^\ell
\xiprime_1|\}, \quad |a|\;=\; \min\{\epsilon^\ell |\xi_1| \; , \;
|\epsilonprime^\ell \xiprime_1|\}.
$$
Since $|A-a|\le 2|A|$ and since $|b|\le |B|/2$ because of ($\ref{Equation:nGrand}$), we have $|B-b|\ge |B|/2$, so we get 
$$
|y|\;\le\; 4 \frac{|A|}{|B|}\cdotp
$$
We will consider the two cases corresponding to the possible
signs of $\ell$, (remember that $n$ is positive).

{\bf First case.} Let $ \ell \le 0$. We have
$$
|A|\;\le\; \Newcst{kappaellpositifmajA} \epsilon^{|\ell|/2}
k^{\cst{kappaxi}}.
$$
We deduce from ($\ref{Equation:Y}$)
\begin{equation}\label{Equation:majyellnegatif} 
1\;\le\; |y|\le 4 \left|\frac{ \xiprime_1}{\alpha}\right | 
\epsilon^{(|\ell |/2)-n} \;\le\; \Newcst{kappa9prime}
\epsilon^{(|\ell |/2)-n} k^{\cst{kappaxi}}.
\end{equation}
Hence there exists $\Newcst{kappa8bis} $ such that
$$
0\;\le\; \log |y| \;\le\; \left(\frac{ |\ell| }{2}-n \right)
\log\epsilon +\cst{kappa8bis} \log k,
$$
from which we deduce the inequality
\begin{equation}\label{Equation:majnellnegatif} 
n\le \frac{ |\ell| }{2}+ \Newcst{kappa9} \log k,
\end{equation}
which will prove useful: $n$ is roughly bounded by $|\ell|$. From
($\ref{Equation:PremierConjugue}$) we deduce the existence of a
constant $\Newcst{kappa10}$ such that
\begin{equation}\label{Equation:ellNegatifMajx} 
|x|\; \le \;\epsilon^{-n/2} |\alphaprime y| + \cst{kappa10}
k^{\cst{kappaxi}} \epsilon^{|\ell |/2}.
\end{equation}

{\bf Second case.} Let $ \ell > 0$. We have
$$
|A|\;\le\; \Newcst{kappaellnegatifmajA} \epsilon^{\ell}
k^{\cst{kappaxi}}.
$$
We deduce from ($\ref{Equation:Y}$) the upper bound
\begin{equation}\label{Equation:ellPosittifMajy} 
1\;\le \;|y|\;\le \;4 \left|\frac{\xi_1}{\alpha}\right |
\epsilon^{\ell-n} \;\le\; \Newcst{kappa13prime} k^{\cst{kappaxi}}
\epsilon^{\ell-n} ;
\end{equation}
hence there exists $\Newcst{kappa8} $ such that
$$
0\;\le\; \log |y|\; \le\; \left(\ell -n \right) \log\epsilon
+\cst{kappa8} \log k.
$$
Consequently,
\begin{equation}\label{Equation:ellPositifMajn}
n\;\le\; \ell+\Newcst{kappaellpositif} \log k.
\end{equation}
From the relation ($\ref{Equation:PremierConjugue}$) we deduce the
existence of a constant $\Newcst{kappa11} $ such that
\begin{equation}\label{Equation:ellPositifMajx} 
1\;\le\; |x|\;\le\; \epsilon^{-n/2} |\alphaprime y| +\cst{kappa11}
k^{\cst{kappaxi}} \epsilon^{- \ell/ 2}.
\end{equation}

By taking into account the inequalities
($\ref{Equation:majyellnegatif}$), ($\ref{Equation:majnellnegatif}$)
and ($\ref{Equation:ellNegatifMajx}$) in the case $\ell\le 0$,
and the inequalities ($\ref{Equation:ellPosittifMajy}$),
($\ref{Equation:ellPositifMajn}$) and
($\ref{Equation:ellPositifMajx}$) in the case $\ell > 0$, let us
show that the existence of a constant $\Newcst{kappa:ellmajore} $
satisfying $|\ell|\le\cst{kappa:ellmajore} \log k$ allows to conclude
the proof of Theorem $\ref{Theoreme:principal}$. As a matter of
fact, suppose
\begin{equation}\label{Equation:nellmajore}
|\ell|\le\cst{kappa:ellmajore} \log k.
\end{equation}
Then ($\ref{Equation:majnellnegatif}$) and
($\ref{Equation:ellPositifMajn}$) imply $n\le
\Newcst{kappanplus1}\log k$, whereupon $|\ell|$ and $n$ are
effectively boun\-ded by a constant times $\log k$. This implies that the elements
$\epsilon^t$, with $t$ being $(|\ell|/2)-n$, $\ell-n$, $-n/2$, $|\ell|/2$ or $-\ell/2$, appearing in ($\ref{Equation:majyellnegatif}$),
($\ref{Equation:ellPosittifMajy}$),
($\ref{Equation:ellNegatifMajx}$) and
($\ref{Equation:ellPositifMajx}$) are boun\-ded from above by
$k^{\Newcst{kappa21bis}}$ for some constant $\cst{kappa21bis}$.
Therefore the upper bound of $|y|$ in the conclusion of Theorem
$\ref{Theoreme:principal}$ follows from
($\ref{Equation:majyellnegatif}$) and
($\ref{Equation:ellPosittifMajy}$) and the upper bound of $|x|$ is a
consequence of ($\ref{Equation:ellNegatifMajx}$) and
($\ref{Equation:ellPositifMajx}$). Our goal is to show that sooner
or later, we end up with the inequality
$(\ref{Equation:nellmajore})$.

 In the case $\ell> 0$, the lower bound $|x|\ge 1$ provides an
extra piece of information. If the term $ \epsilonprime^\ell
\xiprime_1$ on the right hand side of
($\ref{Equation:PremierConjugue}$) does not have an absolute value
$<1/2$, then the upper bound ($\ref{Equation:nellmajore}$) holds
true and this suffices to claim the proof of Theorem
$\ref{Theoreme:principal}$. Suppose now $ |\epsilonprime^\ell
\xiprime_1|<1/2$. Since the relation
($\ref{Equation:ellPosittifMajy}$) implies
$$
\epsilon^{-n/2} | \alphaprime y|\; \le\; 4 \left|
\frac{\xi_1\alphaprime}{\alpha} \right | \epsilon ^{\ell -(3n/2)} ,
$$
we have
$$
1\;\le\; |x|\;\le\; 4\left|\frac{\xi_1\alphaprime}{\alpha} \right |
\epsilon ^{\ell -(3n/2)} + \frac{1}{2}
$$
and
$$
1\;\le\; 8\left|\frac{\xi_1\alphaprime}{\alpha} \right | \epsilon
^{\ell -(3n/2)}.
$$
We deduce
\begin{equation}\label{troisdemiden}
 \frac{3}{2}n\;\le\; \ell + \cst{kappa14} \log k.
\end{equation}
The upper bound in ($\ref{troisdemiden}$) is sharper than the one
in ($\ref{Equation:ellPositifMajn}$), but, amazingly, we used
($\ref{Equation:ellPositifMajn}$) to establish
($\ref{troisdemiden}$).

When $\ell<0$, we have $|\ell-n|\;=\;n+|\ell|\ge |\ell|$, while in the
case $\ell \ge 0$ we have
$$
|\ell-n|\;\ge\; \frac{1}{3} \ell + \frac{2}{3} \ell -n\;\ge\;
\frac{1}{3} |\ell| - \cst{kappa14} \log k,
 $$
because of ($\ref{troisdemiden}$). Therefore, if $\ell$ is
positive (recall ($\ref{troisdemiden}$)), zero or negative (recall ($\ref{Equation:majnellnegatif}$)), we always have
\begin{equation}\label{Equation:majoration-n}
n \;\le\; \frac{2}{3}|\ell|+ \cst{kappa15} \log k
\quad\hbox{and}\quad |\ell-n|\;\ge\; \frac{1}{3} |\ell| -
\cst{kappa14} \log k
\end{equation}
with $\Newcst{kappa14}>0$ and $\Newcst{kappa15}>0$.


\section{Diophantine tool}\label{S:OutilDiophantien}

 Let us remind what we mean by the absolute logarithmic height $\rmh(\alpha)$ of an
 algebraic number $\alpha$ (cf.~\cite{GL326}, Chap. 3).
 For $L$ a number field and for $\alpha \in L$, we define
$$
 \rmh(\alpha)\;=\;\frac{1}{[L:Q]} \log \, H_L(\alpha),
 $$
with
 $$
 H_L(\alpha)\;=\; \prod_{\nu } \max\{1,|\alpha |_\nu
 \}^{d_\nu} 
 $$
 where $\nu$ runs over the set of places of $L$, with $d_\nu$ being the local degree of the place $\nu$ if $\nu$ is ultrametric, 
$d_\nu \;=\; 1$ if $\nu$ is real, $d_\nu\;=\;2$ if $\nu$ is complex. 
 When $f(X)\in\Z[X]$ is the minimal polynomial of $\alpha$ and $f(X)\;=\;a_0\displaystyle\prod_{1\leq j\leq d}
 (X-\alpha_j)$, with $\alpha_1\;=\;\alpha$, it happens that
 $$
 \rmh(\alpha)\;=\; \frac{1}{d} \log\,
 M(f) \quad \mbox{ with }\quad
 M(f)\;=\;|a_0|\prod_{1\leq j \leq d} \max\{1, |\alpha_j|\}.
 $$

We will use two particular cases of Theorem 9.1 of
\cite{GL326}. The first one is a lower bound for the linear form of
logarithms $\;b_0\lambda_0+b_1\lambda_1+b_2\lambda_2$,\, and the second one is a
lower bound for $ \, \gamma_1^{b_1} \gamma_2^{b_2}-1$. \,Here is the
first one.\medskip

\begin{proposition}\label{Proposition:FormeLineaireLogarithmes1}
There exists an explicit absolute constant $c_0>0$ with the following
property. Let $\lambda_0, \lambda_1,\lambda_2$ be three logarithms of algebraic
numbers and let $b_0,b_1,b_2$ be three rational integers such that
$\Lambda\;=\;b_0\lambda_0+b_1\lambda_1+b_2\lambda_2$ be nonzero. Write
$$
\gamma_0\;=\;e^{\lambda_0},\quad \gamma_1\;=\;e^{\lambda_1},\quad \gamma_2\;=\;e^{\lambda_
2}\quad\hbox{and}\quad D\;=\;[\Q(\gamma_0,\gamma_1,\gamma_2) : \Q].
$$
Let $A_0$, $A_1$, $A_2$ and $B$ be real positive numbers satisfying
$$
\log A_i\;\ge \; \max\left\{ \rmh(\gamma_i), {|\lambda_i|\over D} ,
{1\over D}\right\} \qquad(i\;=\;0,1,2)
$$
and
$$
 B\;\ge \max\left\{
 e,\; 
 \;  D,
 \; {|b_2|\over D\log A_0}+ {|b_0|\over D\log A_2},\; 
 \; {|b_2|\over D\log A_1}+ {|b_1|\over D\log A_2}
 \right\}. 
$$
Then
$$
|\Lambda|\;\ge\;\exp\bigl\{- c_0 D^5 (\log D) (\log A_0) (\log A_1) (\log A_2)(\log
B)\bigr\}.
$$
\end{proposition}

The second particular case of Theorem 9.1 in \cite{GL326} that we will use is the next Proposition $\ref{Proposition:FormeLineaireLogarithmes2}$.
It also follows from Corollary 9.22 of \cite{GL326}. We could as well deduce it from Proposition $\ref{Proposition:FormeLineaireLogarithmes1}$.

\begin{proposition}\label{Proposition:FormeLineaireLogarithmes2}
Let $D$ be a positive integer. There exists an explicit constant $c_1>0$, depending only on $D$ with the following property.
Let $K$ be a number field of degree $\le D$. Let $\gamma_1, \gamma_2$ be nonzero elements in $K$ and let $b_1,b_2$ be rational integers. Assume $
\gamma_1^{b_1} \gamma_2^{b_2}\not=1$.
Set
$$
B\;=\;\max\{2, \; |b_1|, |b_2|\}
\quad\hbox{ and, for $i\;=\;1,2$, } \quad
A_i\;=\;\exp\bigl( \max\{e,\; \rmh(\gamma_i)\} \bigr).
$$
Then
$$
|\gamma_1^{b_1} \gamma_2^{b_2} -1|\ge \exp\bigl\{ -
c_1 (\log B)(\log A_1) (\log A_2)\bigr\}.
$$
\end{proposition}

 Proposition $\ref{Proposition:FormeLineaireLogarithmes2}$
 will come into play via its following
 consequence.

\begin{corollaire}\label{Corollaire:DioSinus}
Let $\delta_1$ and $\delta_2$ be two real numbers in the interval 
$[0,2\pi)$. Suppose that the numbers $e^{i\delta_1}$ and
$e^{i\delta_2}$ are algebraic. There exists an explicit 
constant $c_2>0$, depending only upon
$\delta_1$ and $\delta_2$,
 with the following property: for each $n\in\Z$ such
that $\delta_1+n\delta_2\not\in\Z\pi$, we have
$$
|\sin(\delta_1+n\delta_2)|\;\ge\; (|n|+2)^{-c_2}.
$$
\end{corollaire}

\begin{proof}
Write $\gamma_1\;=\;e^{i\delta_1}$ and $\gamma_2\;=\; e^{i\delta_2}$. By
hypothesis, $\gamma_1$ and $\gamma_2$ are algebraic with
$\gamma_1\gamma_2^n\not=1$. Let us use Proposition
$\ref{Proposition:FormeLineaireLogarithmes2}$ with $b_1\;=\;1$,
$b_2\;=\;n$. The parameters $A_1$ and $A_2$ depend only upon $\delta_1$ and $\delta_2$ and
the number $B\;=\;\max\{2,|n|\}$ is bounded from above by $ |n|+2$. Hence
$$
|\gamma_1\gamma_2^n-1|\;\ge\; (|n|+2)^{-c_3}
$$
where $c_3$ depends only upon $\delta_1$ and $\delta_2$.
Let $\ell$ be the nearest integer to $(\delta_1+n\delta_2)/\pi$
 (take the floor if there are two possible values)
 and let $t\;=\;\delta_1+n\delta_2-\ell\pi $. This real number
$t$ is in the interval $(-\pi/2,\pi/2]$. 
Now
 $$ |e^{i t}+1|\;=\; |1+ \cos(t) +i\; \sin(t)|
 \;=\;\sqrt{2(1+ \cos (t))} \ge \sqrt{2}.
 $$
Since $e^{i t}\;=\;(-1)^\ell \gamma_1\gamma_2^n$, we deduce
 \begin{eqnarray*}
|\sin(\delta_1+n\delta_2)|& =& |\sin (t)|\; =\;	
 \frac{1}{2} \left|(-1)^{2\ell}e^{2i t}-1\right|\\[1mm]
&=& \frac{1}{2} \left|(-1)^{\ell}e^{i t}+1\right|\cdot
\left|(-1)^{\ell}e^{i t}-1\right|\;
 \ge\; \frac{\sqrt{2}}{2} \left|\gamma_1\gamma_2^n-1\right|.
\end{eqnarray*}
This secures the proof of Corollary $\ref{Corollaire:DioSinus}$.
\end{proof}

The following elementary lemma makes clear that $e^t\sim 1$
for $t\rightarrow 0$. The first (resp. second) part follows from
Exercice 1.1.a (resp. 1.1.b\, or\, 1.1.c) of \cite{GL326}. We will
use only the second part; the first one shows that the number $t$
in the proof of Corollary $\ref{Corollaire:DioSinus}$ is close to
$0$, but we did not need it.

 \begin{lemme}\label{lemme:GL326exo11b}
\null$\phantom{.}$ (a)
 For $t\in \C$, we have
 $$
 |e^t-1|\;\le\; |t|\max\{1,|e^t|\}.
 $$

 (b) If a complex number $z$ satisfies $|z-1|<1/2$, then there exists
 $t\in\C$ such that $e^t\;=\;z$ and $| t | \le 2|z-1|$. This $t$ is
 unique and is the principal determination of the logarithm of $z$:
 $$
 |\log z|\;\le\; 2 |z-1|.
 $$
 \end{lemme}

\section{ Proof of Theorem $\ref{Theoreme:principal}$}\label{S:DemonstrationThmPpal}

Let us define some real numbers $\theta$, $\delta$ and $\nu$ in the interval 
$[0,2\pi)$ by
$$
\epsilonprime\;=\;\frac{1}{\epsilon^{1/2}} e^{i\theta}, \quad
\alphaprime\;=\; |\alphaprime|e^{i\delta}, \quad \xiprime_1\;=\;|
\xiprime_1|e^{i\nu}.
$$
By ordering the terms of ($\ref{Equation:EquationAuxUnites}$),
we can write this relation as
$$
T_1+T_2+T_3\;=\;0,
$$
and the three terms involved are
 $$\left\{
\begin{array}{lcl}
T_1:\;=\; \epsilon^\ell \xi_1 (\alphaprime
\epsilonprime^n-\alphaprimebar \;
 \epsilonprimebar^n) &=&2i\xi_1 |\alphaprime| \epsilon^{\ell - n/2} \sin (\delta+ n\theta),\\[2mm]
 T_2:\;=\; \alpha\epsilon^n(\epsilonprimebar^\ell \xiprimeunbar - \epsilonprime^\ell \xiprime_1 )
 &=& - 2i | \xiprime_1| \alpha \epsilon^{n-\ell/2} \sin (\nu+ \ell\theta),\\[2mm]
T_3:\;=\; \xiprime_1 \epsilonprime^\ell\alphaprimebar \;
\epsilonprimebar^n- \xiprimeunbar \epsilonprimebar^\ell
\alphaprime\epsilonprime^n &=& 2i | \xiprime_1 \alphaprime|
\epsilon^{-(n+\ell)/2} \sin (\nu-\delta+(\ell-n)\theta).
\end{array}\right.
$$
It turns out that these three terms are purely imaginary. 
We write this zero sum as
$$a+b+c\;=\;0\;\mbox{ with }\; |a|\,\ge\, |b|\,\ge \,|c|,
$$
and we use the fact that this implies that $|a|\le 2|b|$. Thanks
to ($\ref{Equation:majoration-n}$), Corollary
$\ref{Corollaire:DioSinus}$ shows that a lower bound of the sinus
terms is $|\ell|^{ - \Newcst{kappa19}}$ (and an obvious upper
bound is $1$!). Moreover,

\noindent -- The $T_1$ term contains a constant factor and the
factors:

$\bullet$ $|\xi_1|$ with $k^{-\cst{kappaxi}}\leq|\xi_1|\leq k^{\cst{kappaxi}}$,

$\bullet$
 $\epsilon^{\ell-(n/2)}$ (which is the main term),

 $\bullet$
 a sinus with a parameter $n$ (a
 lower bound of the absolute value of that sinus being $n^{-\Newcst{kappasinn}}$).

\noindent -- Similarly, $T_2$ contains a constant factor and the
factors:

$\bullet$ $|\xiprime_1|$, with $k^{-\cst{kappaxi}}\leq |\xiprime_1|\leq
k^{\cst{kappaxi}}$,

$\bullet$
 $\epsilon^{n-(\ell/2)}$ (which the main term),

 $\bullet$
 a sinus with a parameter $\ell $ (a
 lower bound of the absolute value of that sinus being $|\ell|^{-\Newcst{kappasinell}}$).

\noindent -- Similarly, $T_3$ contains a constant factor and
the factors:

$\bullet$ $|\xiprime_1|$, with $k^{-\cst{kappaxi}}\leq |\xiprime_1|\leq k^{\cst{kappaxi}}$,

$\bullet$
 $\epsilon^{-(n+\ell)/2}$ (which the main term),

 $\bullet$
 a sinus with a parameter $\ell-n$ (a
 lower bound of the absolute value of that sinus being
 $|\ell -n|^{-\Newcst{kappasinellmoinsn}}$).


We will consider three cases, and we will use the inequalities
($\ref{Equation:xiun}$) and ($\ref{Equation:majoration-n}$).
This will eventually allow us to conclude that there is an
upper bound for $|\ell|$ and $n$ by an effective constant times $\log k$.
\medskip

\indent {\bf First case}. If the two terms $a$ and $b$ with the
largest absolute values are $T_1$ and $T_2$,
from the inequalities $|T_1|\le 2|T_2|$ and $|T_2|\le 2|T_1|$
(which come from $|b|\le |a|\le 2|b|$),
 we deduce (thanks to ($\ref{Equation:majoration-n}$))
$$
k^{-\Newcst{kappapremiercas1}} |\ell|^{-\Newcst{kappa20}}\; \le\;
\epsilon^{\frac{3}{2}(\ell-n)}\le k^{\Newcst{kappapremiercas2}}
|\ell |^{\Newcst{kappa21}},
$$
whereupon, thanks again to ($\ref{Equation:majoration-n}$), we have
$$
-\Newcst{kappa22} \log k + \frac{|\ell |}{3}\;\le \;|\ell-n|\; \le\;
\Newcst{kappa23}\log|\ell | + \Newcst{kappa24} \log k,
$$
which leads to $|\ell| \le \Newcst{kappa25} (\log k+ \log |\ell|)$.
 This secures the upper bound ($\ref{Equation:nellmajore}$), and ends the
 proof of Theorem $\ref{Theoreme:principal}$.
 \medskip

 \indent
{\bf Second case}. Suppose that the two terms $a$ and $b$ with
the largest absolute values are $T_1$ and $T_3$. By writing\,\,
$|T_1|\le 2|T_3|$ \,\,and \,\, $|T_3|\le 2|T_1|$,\, we obtain
(thanks to ($\ref{Equation:majoration-n}$))
$$
k^{-1/3} |\ell|^{- \Newcst{kappa27}} \;\le\;
 \epsilon^{3\ell/2}
\; \le\;
 k^{1/3} |\ell |^{ \Newcst{kappa28}},
$$
hence
$$
 |\ell|\;\le\; \Newcst{kappa29} ( \log k +\log|\ell |).
$$
Once more, we have $\epsilon^{|\ell|}\le k^{ \Newcst{kappa26}}$,
and we saw that the upper bound
 ($\ref{Equation:majoration-n}$) allows to draw the conclusion.
\medskip

 \indent
{\bf Third case}. Let us consider the remaining case, namely,
the two terms $a$ and $b$ with the largest absolute values being
$T_2$ and $T_3$.
 Consequently, in the relation $T_1+T_2+T_3\;=\;0$, written in
 the form $a+b+c\;=\;0$ with $|a|\ge |b|\ge |c|$, we have $c\;=\;T_1$.
Writing $|T_2|\le 2|T_3|$ and $|T_3|\le 2|T_2|$, we obtain
$$
 k^{-1/3} |\ell|^{- \Newcst{kappa30}}\;\le\; \epsilon^{3n/2}\le k^{1/3} |\ell |^{ \Newcst{kappa31}}.
$$
From the second of these inequalities, we deduce the existence of $ \Newcst{kappa32}$ such that
\begin{equation}\label{Equation:Majoration-n}
n \;\le\; \cst{kappa32}(\log k + \log|\ell | ).
\end{equation}

\indent {\tt Remark.} {\rm The upper bound
($\ref{Equation:Majoration-n}$) allows to proceed as in the usual
proof of the Thue theorem where $n$ is fixed.}
\medskip

From the upper bound \,$|T_1|\le |T_2|$,\, one
deduces \,$n>\ell- \Newcst{kappa33}\log k$, \,so that
($\ref{Equation:Majoration-n}$) leads right away to the conclusion
if $\ell$ is positive.

Let us suppose now that $\ell$ is negative. Let us consider again
the equation ($\ref{Equation:EquationAuxUnites}$) that we write in
the form
\begin{equation}\label{Equation:Unites} 
\rho_n \epsilon^\ell + \mu_n \epsilonprime^\ell -\mubar_n
\epsilonprimebar^\ell\;=\;0
\end{equation}
with
$$
\rho_n\;=\; \xi_1 (\alphaprime \epsilonprime^n-\alphaprimebar \;
\epsilonprimebar^n) \quad \hbox{and}\quad \mu_n\;=\; \xiprime_1
(\alphaprimebar \; \epsilonprimebar^n-\alpha\epsilon^n).
$$
We check (cf. Property 3.3 of \cite{GL326})
$$
\rmh(\mu_n)\le \Newcst{kappa34} (n +\log k).
$$
Let us divide each side of ($\ref{Equation:Unites}$) by $-\mu_n
\epsilonprime^\ell$:
$$
\frac{\mubar_n \epsilonprimebar^\ell}{\mu_n
\epsilonprime^\ell}-1\;=\;\frac{\rho_n \epsilon^\ell}{\mu_n
\epsilonprime^\ell}\cdotp
$$
We have
$$
|\alpha' \epsilonprime^n-\alphaprimebar\,\,
\epsilonprimebar^n|\;\leq\;
 \, |\alpha' \epsilonprime^n|+ | \alphaprimebar\,\,
\epsilonprimebar^n |
\, =
\, 
 2\left| \epsilonprime^n \alphaprime\right|
$$
and, using ($\ref{Equation:nGrand}$),
$$
|\alphaprimebar\,\, \epsilonprimebar^n -\alpha \epsilon^n|
\;\geq\; \frac{1}{2}|\alpha|\epsilon^n.
$$
Since
$$\house{\xi_1}\;\le k^{\cst{kappaxi}}
\; \mbox{ and } \;|\xiprime_1|\;>\;k^{-\cst{kappaxi}}
$$
by ($\ref{Equation:xiun}$), we come up with
$$
|\rho_n|\;\le\; \Newcst{kappa35} k^{\cst{kappaxi}} \epsilon^{ n /2},\quad
|\mu_n|\;\ge\; \Newcst{kappa36}
 \epsilon^{n} k^{-\cst{kappaxi}}.
$$
Therefore, since $|\epsilon'|^{-1}\;=\;\epsilon^{1/2}$, we have
 \begin{equation}\label{Equation:majorationFormeLineaire} 
\left| \frac{\mubar_n \epsilonprimebar^\ell}{\mu_n
\epsilonprime^\ell}-1\right|\;=\; \left| \frac{\rho_n
\epsilon^\ell}{\mu_n \epsilonprime^\ell}\right|\;\leq\;
\Newcst{kappa37}
 \epsilon^{-(n+3|\ell|)/2}
 k^{\cst{kappaxi}}.
\end{equation}
\smallskip
We denote by $\log$ the principal value of the logarithm and we set 
 $$
\lambda_1\;=\; \log \left(\frac{\epsilonprimebar}{\epsilonprime}\right),
\quad
\lambda_2\;=\;
 \log \left(\frac{\mubar_n }{\mu_n}\right)
\quad\hbox{and}
\quad
\Lambda\;=\;\log \left(
 \frac{\mubar_n \epsilonprimebar^\ell}{\mu_n
\epsilonprime^\ell}
\right).
$$
We have
$$
\lambda_1\;=\;2i\pi \nu \quad \lambda_2\;=\;2i\pi\theta_n,
$$
where $\nu$ and $\theta_n$ are the real numbers in the interval 
 $[0,1)$ defined by
$$
\frac{\epsilonprimebar}{\epsilonprime}\;=\;e^{2i\pi \nu}
\quad
\hbox{and}\quad 
\frac{\mubar_n}{\mu_n}\;=\;e^{2i\pi \theta_n}.
$$
From $ e^\Lambda\;=\;e^{\ell\lambda_1+\lambda_2}$ we deduce $\Lambda- \ell \lambda_1-\lambda_2\;=\;2i\pi h$ with $h\in \Z$. From Lemma $\ref{lemme:GL326exo11b}b$ we deduce 
$|\Lambda|\le 2 |e^\Lambda-1|$. Using $|\Lambda|<2\pi$ and writing 
$$
2i\pi h\;=\;\Lambda-2i\pi\ell\nu-2i\pi \theta_n,
$$
we deduce $|h|\le |\ell|+2$. 

 In Proposition
$\ref{Proposition:FormeLineaireLogarithmes1}$, let us take 
$$
b_0\;=\;h, \quad b_1\;=\;\ell, \quad b_2\;=\;1, \quad \gamma_0=1,\quad \lambda_0\;=\;2i\pi, 
\qquad 
\gamma_1\;=\;\frac{\epsilonprimebar}{\epsilonprime}, \quad 
\gamma_2\;=\;\frac{\mubar_n}{\mu_n},
$$
$$
A_0\; = \; A_1\;=\;\Newcst{38}, \quad
A_2\;=\;(k \, \, \epsilon^n)^{\Newcst{kappa51} }, \quad
B\;=\; e +\frac{ |\ell | }{\log A_2}\;\cdotp
 $$
 Notice that the degree $D$ of the field $\Q(\gamma_0, \gamma_1, \gamma_2)$ is $\le 6$. 
Then we obtain
$$
\left| \frac{\mubar_n }{\mu_n} \left( \frac{\epsilonprimebar}{
\epsilonprime}\right)^\ell -1\right|\;=\; |e^\Lambda-1|\; \ge\;
\frac{1}{2} |\Lambda| \;\ge \; \exp\bigl\{ - \Newcst{kappa41} (\log
A_2)(\log B)\big\}.
$$
By combining this estimate with
($\ref{Equation:majorationFormeLineaire}$), we deduce
$$
|\ell |\;\le \;\Newcst{kappa46} (n+\log k )\log B,
$$
which can also be written as $B\le \Newcst{kappa47} \log B$, hence
$B$ is bounded. This allows to obtain
 $$
 | \ell | \;\le\; \Newcst{kappa48} (n+ \log k ).
 $$
 We use ($\ref{Equation:Majoration-n}$) to deduce $\epsilon^{|\ell|}\le k^{ \cst{kappa26}}$
 and we saw that the upper bound ($\ref{Equation:nellmajore}$)
 leads to the conclusion of the main Theorem
 $\ref{Theoreme:principal}$.

\vfill

\end{document}